\documentclass[journal]{IEEEtran}
\ifCLASSINFOpdf
\else
\fi
\usepackage{amssymb}
\usepackage{amsmath}

\usepackage{bbding}
\usepackage{amsfonts}
\usepackage{cite}
\usepackage{cases}
\usepackage{subfigure}
\usepackage{algorithmic}
\usepackage{mathrsfs}
\usepackage{bm}
\usepackage{verbatim}
\usepackage{enumerate}
\usepackage{booktabs}
\usepackage{multirow}
\usepackage[linesnumbered,ruled,longend]{algorithm2e}
\usepackage{textcomp}
\usepackage{slashbox}
\usepackage{graphics}
\usepackage{graphicx}
\usepackage{subeqnarray}
\usepackage{cases}

\usepackage{url}

\usepackage{array}
\newcommand{\PreserveBackslash}[1]{\let\temp=\\#1\let\\=\temp}
\newcolumntype{C}[1]{>{\PreserveBackslash\centering}p{#1}}
\newcolumntype{R}[1]{>{\PreserveBackslash\raggedleft}p{#1}}
\newcolumntype{L}[1]{>{\PreserveBackslash\raggedright}p{#1}}

\def\ba{\begin{array}}
\def\ea{\end{array}}
\newcommand{\beq}{\begin{equation}}
\newcommand{\eeq}{\end{equation}}
\newcommand{\bq}{\begin{eqnarray}}
\newcommand{\eq}{\end{eqnarray}}
\newcommand{\bqn}{\begin{eqnarray*}}
\newcommand{\eqn}{\end{eqnarray*}}
\newcommand{\bee}{\begin{enumerate}}
\newcommand{\eee}{\end{enumerate}}
\newcommand{\bi}{\begin{itemize}}
\newcommand{\ei}{\end{itemize}}

\newtheorem{lemma}{\textbf{Lemma}}
\newtheorem{theorem}{\textbf{Theorem}}
\newtheorem{remark}{\textbf{Remark}}

\IEEEoverridecommandlockouts

\usepackage{colortbl}

\usepackage{comment}
\newboolean{showcomments}
\setboolean{showcomments}{true}
\newcommand{\slow}[1]{\ifthenelse{\boolean{showcomments}}
{ \textcolor{red}{(SL:  #1)}}{}}
\newcommand{\you}[1]{\ifthenelse{\boolean{showcomments}}
{ \textcolor{blue}{(PCY:  #1)}}{}}

\begin{document}

\title{Scheduling of EV Battery Swapping, II:  Distributed Solutions}

\author{Pengcheng~You,~\IEEEmembership{Student~Member,~IEEE,}
        Steven~H.~Low,~\IEEEmembership{Fellow,~IEEE,}\\
        Liang~Zhang,~\IEEEmembership{Student~Member,~IEEE,}
        Ruilong~Deng,~\IEEEmembership{Member,~IEEE,}\\
        Georgios~B.~Giannakis,~\IEEEmembership{Fellow,~IEEE,}
        Youxian~Sun,
        and Zaiyue~Yang,~\IEEEmembership{Member,~IEEE}
\thanks{P. You, Y. Sun and Z. Yang are with the State Key Laboratory of Industrial Control Technology, Zhejiang University, Hangzhou, 310027, China (e-mail: pcyou@zju.edu.cn; yxsun@iipc.zju.edu.cn; yangzy@zju.edu.cn).}
\thanks{P. You and S. H. Low are with the Engineering and Applied Science Division, California Institute of Technology, Pasadena, CA 91125 USA (e-mail: pcyou@caltech.edu; slow@caltech.edu).}
\thanks{R. Deng is with the Department of Electrical and Computer Engineering, University of Alberta, Edmonton, AB, Canada T6G 1H9 (e-mail: ruilong@ualberta.ca).}
\thanks{L. Zhang and G. B. Giannakis are with the Department of Electrical and Computer Engineering, University of Minnesota, Minneapolis, MN 55455 USA (e-mail: zhan3523@umn.edu; georgios@umn.edu).}
}

\maketitle

\begin{abstract}
In Part I of this paper we formulate an optimal scheduling problem for battery swapping that
assigns to each electric vehicle (EV) a best station to swap its depleted
battery based on its current location and state of charge.
The schedule aims to minimize a weighted sum of total travel distance and generation
cost over both station assignments and power flow variables,
subject to EV range constraints, grid operational constraints and AC power flow equations.
We propose there a centralized solution based on the second-order cone programming (SOCP) relaxation of optimal power flow (OPF) and generalized Benders decomposition
that is suitable when global information is available.  In this paper we propose two
distributed solutions based on the alternating direction method of multipliers (ADMM) and dual decomposition respectively that are
suitable for cases where the distribution grid, battery stations and EVs are managed
by separate entities.  Our algorithms allow these entities to make individual decisions
but coordinate through privacy-preserving information exchanges to jointly solve an
approximate version of the joint battery swapping scheduling and OPF problem.
We evaluate our algorithms through simulations.
\end{abstract}


\IEEEpeerreviewmaketitle

\section{Introduction}

\subsection{Motivation}

In Part I of this paper we formulate an optimal scheduling problem for battery swapping that
assigns to each EV a best station to swap its depleted
battery based on its current location and state of charge. The station assignment not only determines EVs' travel distance, but can also impact significantly the
power flows on a distribution network because batteries are large loads.
The schedule aims to minimize a weighted sum of total travel distance and generation cost over both station assignments and power flow variables,
subject to EV range constraints, grid operational constraints and AC power flow equations.
This joint battery swapping scheduling and OPF problem
is nonconvex and computationally difficult because the AC power flow equations are
nonlinear and the assignment variables are binary.

We propose in Part I a centralized solution based on the SOCP relaxation of
OPF, which deals with the nonconvexity of power flow equations, and generalized
Benders decomposition,
which deals with the binary nature of assignment variables.
When the relaxation of OPF is exact, this
approach computes a global optimum.
It is however suitable only for cases where the
distribution network, battery stations, and EVs are managed centrally by the same
operator, as is the current electric taxi program of State Grid in China.
We expect that, as EVs proliferate and as battery swapping models mature, an equally
(if not more) likely business model will emerge where the distribution grid is managed
by a utility company, battery stations are managed by a station operator (or multiple
station operators), and EVs may be managed by multiple taxi companies or by
individual drivers.  In particular, the set of EVs to be scheduled may include a large number
of private cars in addition to fleet vehicles.  The centralized approach of Part I will not be
suitable for these future scenarios, for two reasons.

First, the operator in the centralized approach needs global information such as the grid
topology, impedances, operational constraints, background loads, availability
of fully-charged batteries at each station, locations and states of charge of EVs, etc.
In the future, the grid, battery stations, and EVs will likely be operated by different entities
that do not share their private information.
Second, generalized Benders decomposition solves a mixed-integer convex
problem in each iteration and is computationally expensive.  It is hard to scale it
to compute in real time an optimal station assignment and an (relaxed) OPF solution in
future scenarios where the numbers of EVs and battery stations are large.
In this paper we aim to develop distributed solutions that preserve private information
and are more suitable for these future scenarios.

Instead of generalized Benders decomposition, we relax the binary assignment variables
to real variables in $[0, 1]$.  With both the SOCP relaxation of OPF and
the relaxation of binary variables, the resulting approximate problem of joint
battery swapping scheduling and OPF is a convex program.   This allows us to develop two
distributed solutions where different entities make their individual decisions
but are coordinated through privacy-preserving information exchanges to jointly
solve the global problem.
The first solution is based on ADMM
and is for cases where the distributed grid is managed by
a utility company and all stations and EVs are managed by a station operator.
Here the utility company maintains a local estimate of some \emph{aggregate}
assignment information that is computed by the station operator, and they
exchange the aggregate information and its estimate to attain consensus.
The second solution is based on dual decomposition and is for cases
where the distributed grid is managed by a utility company, all battery stations by
a station operator, and all EVs are individually operated.
The utility company still sends its local estimate to the station operator while the
station operator does not need to send the utility company the aggregate
assignment information, but only some Lagrange multipliers.  The station
operator also \emph{broadcasts} Lagrange multipliers to all EVs and
individual EVs respond by sending the station operator their choices of stations
for battery swapping based on the Lagrange multiplier values and their
current locations and driving ranges.
In both approaches, given aggregate information and Lagrange multipliers
that are exchanged,  different entities only need their own local states
(e.g., power flow variables) and local data (e.g., impedance values,
available batteries, EV locations and driving ranges) to
iteratively compute their own decisions.   See Fig. \ref{fig:distributed}.
\begin{figure*}[htbp]
\centering
\includegraphics[width=0.6\textwidth]{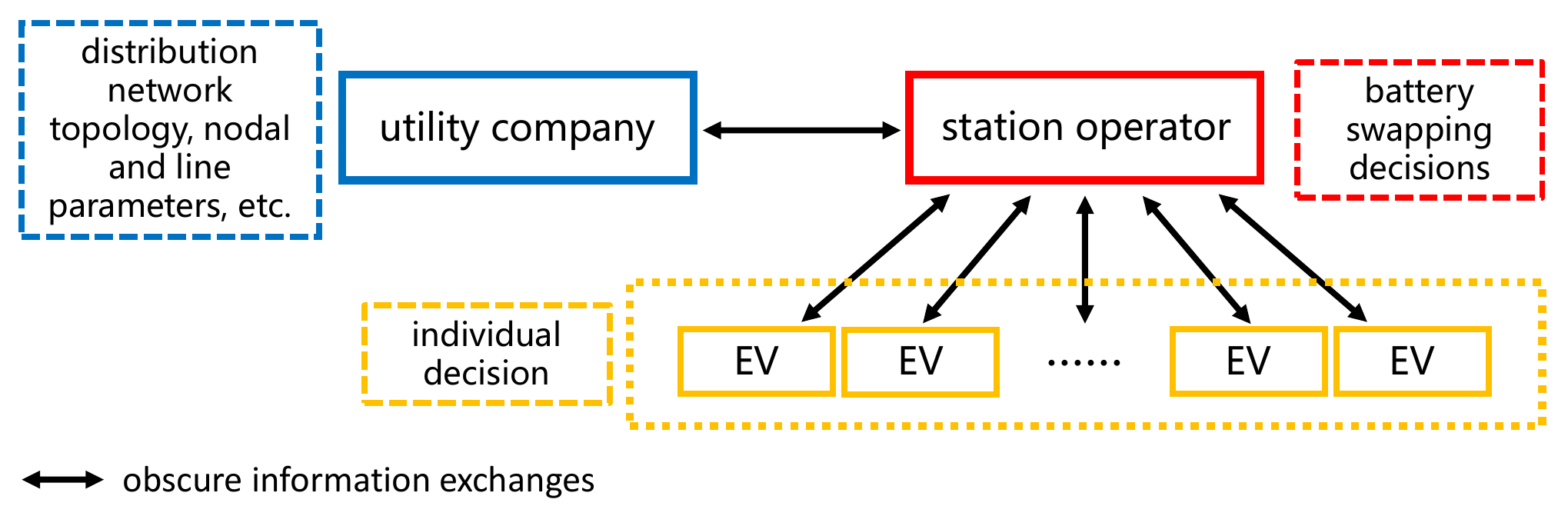}
\caption{Distributed framework. }
\label{fig:distributed}
\end{figure*}
As we will discuss later,
both distributed algorithms are able to converge to a solution, in which the station assignment may be non-binary, to the relaxed version of the joint battery swapping scheduling and OPF problem. However, we show it is easy to discretize the obtained assignment to achieve a binary one that is close to optimum, and the SOCP relaxation is usually exact \cite{Low2014a, Low2014b}.

\subsection{Literature}

The privacy issues in smart grids have drawn much attention from academia due to the vision of more and more interconnections in power systems to strive for strength, security, stability and efficiency \cite{mcdaniel2009security}. Most previous work lays emphasis on the privacy issue of residential loads from the perspective of smart meters \cite{kalogridis2010privacy, efthymiou2010smart, liu2016optimal, yang2014optimal}, and only a small portion of the related literature realizes the significance of EVs' privacy despite their expected high penetration in the future and the resulting giant impact on smart grids. Currently, privacy concerns for EVs mainly arise in Vehicle-to-Grid networks where the status of EVs has to be continuously monitored. Tseng \cite{tseng2012secure} proposes a secure and privacy-preserving communication protocol based on restrictive partially blind signatures to protect EV owners from identity and location information leakage. Liu \emph{et al.} \cite{liu2012aggregated} design an aggregated-proofs based authentication scheme to collect EVs' status without revealing any individual privacy. Besides, Nicanfar \emph{et al.} \cite{nicanfar2013security} present different situations where EVs may be involved in the smart grid context, and provide the corresponding authentication schemes to preserve EV owners' privacy. Nonetheless, the privacy leakage of scheduling information in EV battery charging/swapping is often ignored.

In contrast to centralized algorithms, distributed (decentralized) algorithms inherently preserve privacy as global information is not necessary in local computation. Hence there is a large literature of distributed algorithm design with various applications for privacy-preserving purposes. Liu \emph{et al.} \cite{liu2016optimal} schedule thermostatically controlled devices and batteries in a household to hide its actual load profiles such that no privacy can be inferred from electricity usage. Yang \emph{et al.} \cite{yang2014optimal} design an online control algorithm of batteries that only uses the observations of the current load requirement and electricity price to strike a tradeoff between the smart meter data privacy and the electricity bill for customers. Clifton \emph{et al.} \cite{clifton2002tools} present a toolkit of distributed algorithms that can be combined for specific privacy-preserving data mining applications. Zhou \emph{et~al.} \cite{zhou2015psmpa} devise a multi-level privacy-preserving cooperative authentication scheme to realize different levels of privacy requirement for a distributed m-healthcare cloud computing system that shares personal health information among healthcare providers. Liu \emph{et al.} \cite{liu2016distributed} propose a consensus-based distributed speed advisory system that optimally determines a common speed for a given area in a privacy-aware manner to minimize the group emissions of fuel vehicles or the group battery consumptions of EVs.

\vspace{0.13in}
The remainder of this paper is organized as follows. We revisit the problem formulation in Sec.~\ref{sec:recap}. The proposed distributed solutions via ADMM and dual decomposition are elaborated in Sec.~\ref{sec:solution}, followed by numerical results in Sec.~\ref{sec:numerical}. At last, Sec.~\ref{sec:conclusion} concludes.

\section{Problem formulation}\label{sec:recap}

We now summarize the joint battery swapping scheduling and OPF problem in Part I \cite{you2016EVscheduling1},
using the notations defined there.

An assignment of stations\footnote{Throughout this paper stations refer to battery stations.} to EVs
for battery swapping is represented by the binary variables $u := (u_{aj}, a\in \mathbb A, j\in \mathbb N_w)$ where
\bqn
u_{aj} & = & \left\{
\begin{split}
& 1\quad \mathrm{if~station}~j~\mathrm{is~assigned~to~EV}~a\\
& 0\quad \mathrm{otherwise}
\end{split}
\right.
\eqn
A station assignment must satisfy the following conditions.
\begin{subequations}
\begin{itemize}
\item The assigned station must be in every EV's driving range:
\beq\label{eq:con.tax3}
      u_{aj}d_{aj} \le  \gamma_a c_a,\quad j\in\mathbb{N}_w, a\in\mathbb{A}
\eeq
\item Exactly one station is assigned every EV:
\beq\label{eq:con.tax1}
  \sum\limits_{j\in\mathbb{N}_w}{u_{aj}}=1,\quad a\in\mathbb{A}
\eeq
\item Every assigned station has enough fully-charged batteries for EVs:
\beq\label{eq:con.tax2}
  \sum\limits_{a\in\mathbb{A}}{u_{aj}}\le m_j,\quad j\in\mathbb{N}_w
\eeq
\end{itemize}
\label{eq:con.tax}
\end{subequations}

A station assignment will add loads to the distribution network at buses in $\mathbb N_w$
that supply electricity to stations.  The net power injections
$s_j=p_j+\textbf{i}q_j$ depend on the station assignment $u$ according to
\begin{small}
\begin{subequations}
\bq
p_j & = & \left\{
\begin{split}
&p_j^g-p_j^b-r \left(M_j-m_j+\sum_{a\in\mathbb{A}}{u_{aj}}\right),\quad j\in\mathbb{N}_w\\
&p_j^g-p_j^b,\quad j\in\mathbb{N}/\mathbb{N}_w
\end{split}
\right.
\label{eq:inj.p}
\\
q_j & =  & q_j^g-q_j^b,\quad j\in\mathbb{N}
\label{eq:inj.q}
\eq
\label{eq:inj}%
\end{subequations}\end{small}%
An active distribution network is modeled by the \emph{DistFlow} equations from \cite{baran1989optimals}:
\begin{subequations}\label{eq:df}
\begin{align}\label{eq:df.a}
\sum\limits_{k:(j,k)\in\mathbb{E}}{S_{jk}} &=S_{ij}-z_{ij}l_{ij}+s_j,\quad j\in\mathbb{N}\\\label{eq:df.b}
v_j-v_k &=2 \mathrm{Re} (z_{jk}^HS_{jk})-|z_{jk}|^2l_{jk},\quad j\rightarrow k \in\mathbb{E}\\\label{eq:df.c}
v_jl_{jk} &= |S_{jk}|^2, \quad j\rightarrow k \in\mathbb{E}
\end{align}
\end{subequations}
The power flow quantities must satisfy the following constraints on grid operation:
\begin{subequations}
\begin{itemize}
\item voltage stability
  \beq\label{eq:con.vol}
  \underline{v}_j \le v_j \le\overline{v}_j, \quad j \in \mathbb{N}
  \eeq
\item generation capacity
  \bq\label{eq:con.gen.p}
  \underline{p}^g_j \le p^g_j \le\overline{p}^g_j, \quad j \in \mathbb{N} \\\label{eq:con.gen.q}
  \underline{q}^g_j \le q^g_j \le\overline{q}^g_j, \quad j \in \mathbb{N}
  \eq
\item line transmission capacity
  \beq\label{eq:con.powflow}
  |S_{jk}| \le \overline{S}_{jk}, \quad j\rightarrow k\in\mathbb{E}
  \eeq
\end{itemize}
\label{eq:con.PF}
\end{subequations}

The joint battery swapping scheduling and OPF problem is to minimize a weighted sum of total generation cost in the distribution network
and total travel distance of EVs over both station assignments and power flow variables:
\bq\label{eq:primal problem}
\min\limits_{{u},{s}, {s}^g, \atop {v},{l}, {S}} && \sum\limits_{j\in\mathbb{N}}{f_j(p_j^g)} + \alpha \sum\limits_{a\in\mathbb{A}}\sum\limits_{j\in\mathbb{N}_w} d_{aj} u_{aj} \\\nonumber
\rm{s.t.} &&
\eqref{eq:con.tax}
\eqref{eq:inj}
\eqref{eq:df}
\eqref{eq:con.PF}
\\\nonumber
&& u_{aj} \in \{0, 1\}, \ \ a\in \mathbb A,\, j\in\mathbb N_w
\eq


\section{Distributed solutions}\label{sec:solution}

\subsection{Relaxations}

The joint battery swapping scheduling and OPF problem \eqref{eq:primal problem} is computationally difficult for two reasons.
First, the quadratic equality \eqref{eq:df.c} is nonconvex.
Second, the station assignment variables $u$ are binary.

To deal with the first difficulty, we replace \eqref{eq:df.c} by an inequality, i.e., replace
the \emph{DistFlow} equations \eqref{eq:df} in the problem \eqref{eq:primal problem}
by:
\begin{subequations}\label{eq:df2}
\begin{align}\label{eq:df2.a}
\sum\limits_{k:(j,k)\in\mathbb{E}}{S_{jk}} &=S_{ij}-z_{ij}l_{ij}+s_j,\quad j\in\mathbb{N}\\\label{eq:df2.b}
v_j-v_k &=2 \mathrm{Re} (z_{jk}^HS_{jk})-|z_{jk}|^2l_{jk},\quad j\rightarrow k \in\mathbb{E}\\\label{eq:df2.c}
v_jl_{jk} &\geq |S_{jk}|^2, \quad j\rightarrow k \in\mathbb{E}
\end{align}
\end{subequations}
Fixing any assignment $u\in\{0,1\}^{AN_w}$, the optimization problem is then a convex
problem.  If an optimal solution to the SOCP relaxation
attains equality in \eqref{eq:df2.c} then the solution also satisfies \eqref{eq:df} and
is therefore optimal (for the given $u$).  In this case, we say that the SOCP relaxation
is \emph{exact}.
Sufficient conditions are known that guarantee the exactness of the SOCP relaxation;
see \cite{Low2014a, Low2014b} for a comprehensive tutorial and references therein.
Even when these conditions are not satisfied, SOCP relaxation for practical radial networks
is still often exact, as confirmed also by our simulations in
Sec.~\ref{sec:numerical}.

To deal with the second difficulty, we use generalized Benders decomposition in Part I
\cite{you2016EVscheduling1}.  This approach computes an optimal solution when
SOCP relaxation is exact, but it is computationally expensive as it requires solving
a binary linear program (as well as an SOCP relaxation) in each iteration of the generalized
Benders decomposition procedure.   Moreover, the computation is centralized and is
suitable only when a single organization, e.g., State Grid in China, operates all of
the distribution grid, stations, and EVs.

In this paper, we develop distributed solutions that are suitable for cases where
these three are operated by different organizations that do not share their private
information.
To deal with the second difficulty, we relax the binary variables $u_{aj}$ to real
variables $u_{aj}\in[0,1]$, $a\in\mathbb{A},~j\in\mathbb{N}_w$.
The constraints \eqref{eq:con.tax} are then replaced by:
\begin{subequations}
\bq
\!\!\!\!\!\!\!\!
\label{eq:con.tax3.2}
u_{aj}\in\left[0,1\right], \ u_{aj}=0 \ \mathrm{if}~d_{aj} > \gamma_a c_a,
& \!\!\!\!\!\!\!\! &
j\in\mathbb{N}_w,~a\in\mathbb{A}
\\
\sum_{j\in\mathbb N_w} u_{aj} = 1,
& \!\!\!\!\!\!\!\! &   a\in\mathbb{A}
\label{eq:con.tax1.2}
\\
\sum_{a\in\mathbb A} u_{aj} \leq m_j,
& \!\!\!\!\!\!\!\! &   j\in\mathbb N_w
\label{eq:con.tax2.2}
\eq
\label{eq:con.tax.2}%
\end{subequations}

In summary, in this paper we solve the following convex relaxation of \eqref{eq:primal problem}:
\bq\label{eq:primal problem tr}%
\min\limits_{{u},{s}, {s}^g, \atop {v},{l}, {S}} && \sum\limits_{j\in\mathbb{N}}{f_j(p_j^g)} + \alpha \sum\limits_{a\in\mathbb{A}}\sum\limits_{j\in\mathbb{N}_w} d_{aj} u_{aj} \\\nonumber
\rm{s.t.} &&
\eqref{eq:inj}
\eqref{eq:con.PF}
\eqref{eq:df2}
\eqref{eq:con.tax.2}
\eq
This problem has a convex objective and convex quadratical constraints.
After an optimal solution $(x^*, u^*)$ of \eqref{eq:primal problem tr} is obtained, we
check if $x^*$ attains equality in \eqref{eq:df2.c}.
We also discretize $u_{aj}^*$ into
$\{0, 1\}$, e.g., by setting  for each EV $a$ a single largest $u_{aj}$ to 1 and the rest
to 0.  An alternative is to randomize the station assignment using $u_{aj}^*$ as
a probability distribution. As we will show later, the discretization can be readily implemented and achieve an assignment close to optimum.


\subsection{Distributed solution via ADMM}
\label{subsec:ADMM}

The problem \eqref{eq:primal problem tr} decomposes naturally into two subproblems,
one on station assignments over $u$ and the other on OPF over
$(s, {s}^g,  {v},{l}, {S})$.   The station assignment subproblem will be solved by a station operator
that operates the network of stations.
The OPF subproblem will be solved by a utility company.
Our goal is to design a distributed algorithm for them to jointly solve \eqref{eq:primal problem tr}
without sharing their private information.

These two subproblems are coupled only in
\eqref{eq:inj.p} where the utility company needs the load $r \left(M_j-m_j+ \sum_{a\in\mathbb{A}} u_{aj} \right)$ of station $j$
 in order to compute the net real power injection $p_j$.  This quantity depends on the total number of EVs that each station $j$ is assigned to and is computed by the station operator.
Their computation
 can be decoupled by introducing an auxiliary variable $w_j$ at each bus (station)
$j$ that represents the utility company's estimate of the quantity $r \left(M_j-m_j+ \sum_{a\in\mathbb{A}} u_{aj} \right)$, and
requiring that they be equal at optimality.

Specifically, recall the station assignment variables $u$, and denote the power flow variables by
$x := (w, s, {s}^g,  {v},{l}, {S})$ where $w := (r \left(M_j-m_j+ \sum_{a\in\mathbb{A}} u_{aj} \right), j\in\mathbb N_w)$.   Separate the objective function by defining
\bqn
f(x) & := & 	\sum\limits_{j\in\mathbb{N}}{f_j(p_j^g)}
\\
g(u) & := & 	\alpha
	\sum\limits_{a\in\mathbb{A}}\sum\limits_{j\in\mathbb{N}_w} d_{aj} u_{aj}
\eqn
Replace the coupling constraints \eqref{eq:inj} by constraints local to bus $j$:
\begin{subequations}
\bq
p_j & = & \left\{
\begin{split}
&p_j^g-p_j^b-w_j,\quad j\in\mathbb{N}_w\\
&p_j^g-p_j^b,\quad j\in\mathbb{N}/\mathbb{N}_w
\end{split}
\right.
\label{eq:inj.pw}
\\
q_j & =  & q_j^g-q_j^b,\quad j\in\mathbb{N}
\label{eq:inj.qw}
\eq
\label{eq:tax2injw}%
\end{subequations}
Denote the local constraint set for $x$ by
\bqn
\mathbb X & := & \{ x\in \mathbb R^{(|\mathbb N_w| + 5|\mathbb N|+3|\mathbb E|)} \, :\,
	x \text{ satisfies } \eqref{eq:con.PF} \eqref{eq:df2} \eqref{eq:tax2injw} \}
\eqn
Denote the local constraint set for $u$ by
\bqn
\mathbb U & := & \{ u\in \mathbb R^{A N_w} \, :\,
	u \text{ satisfies } \eqref{eq:con.tax.2} \}
\eqn
To simplify notation, define $u_j := \sum_{a\in\mathbb A} u_{aj}$
for $j\in\mathbb N_w$.
Then the problem \eqref{eq:primal problem tr} is equivalent to
\begin{subequations}
\bq
\min\limits_{x,u} && f(x) + g(u)
\label{eq:PPA.1}
\\
\rm{s.t.} &&  x\in\mathbb{X},~u\in\mathbb{U}
\label{eq:PPA.2}
\\
	& & 	w_j = r \left(M_j-m_j+ u_j \right), \ j\in\mathbb N_w
\label{eq:PPA.3}
\eq\label{eq:primal problem admm}%
\end{subequations}%

We now apply ADMM to \eqref{eq:primal problem admm}.
Let $\lambda$ be the Lagrange multiplier vector corresponding to the coupling constraint
\eqref{eq:PPA.3}, and define the augmented Lagrangian:
\begin{subequations}
\bq
\label{eq:augmentedL.1}
L_{\rho}(x, {u},{\lambda}) & := & f(x) + g(u) +  h_\rho(w, u, \lambda)
\eq
where $h_\rho$ depends on $(x, u)$ only through $(w_j, u_j, j\in\mathbb N_w)$:
\beq
\begin{split}
h_\rho(w, u, \lambda) := &
	\sum\limits_{j\in\mathbb{N}_w}\lambda_j[w_j-r \left(M_j-m_j+ u_j \right)]\\
	& + \ \frac{\rho}{2} \sum\limits_{j\in\mathbb{N}_w}[w_j- r \left(M_j-m_j+ u_j \right)]^2
\label{eq:augmentedL.2}
\end{split}
\eeq
\label{eq:augmentedL}%
\end{subequations}
and $\rho$ is the step size for dual variable $\lambda$ updates.
The standard ADMM procedure is to iteratively and sequentially update $(x, u, \lambda)$:
for $n=0, 1, \dots$,
\begin{small}
\begin{subequations}
\bq
{x}(n+1) &\!\!\!\!\!\! := \!\!\!\!\!\! & \arg\min\limits_{{x}\in\mathbb{X}}\ f(x) + h_\rho(w, {u}(n), {\lambda}(n))
\label{eq:xupdate}
\\
{u}(n+1) & \!\!\!\!\!\! :=  \!\!\!\!\!\! & \arg\min\limits_{{u}\in\mathbb{U}}\ g(u) + h_\rho({w}(n+1), u, {\lambda}(n))
\label{eq:uupdate}
\\
{\lambda}_j(n+1) & \!\!\!\!\!\! := \!\!\!\!\!\! & \lambda_j(n)\, + \, \rho[w_j(n+1)\nonumber\\
& \!\!\!\!\!\! ~~ \!\!\!\!\!\! & - r(M_j-m_j+u_j(n+1))], \ j\in\mathbb{N}_w
\label{eq:lupdate}
\eq
\label{eq:admmupdates}
\end{subequations}\end{small}%
\begin{remark}
\bee
\item The $x$-update \eqref{eq:xupdate} is carried out by the utility company and
	involves minimizing a convex objective with convex quadratic constraints.
	The $(u,\lambda)$-updates \eqref{eq:uupdate}\eqref{eq:lupdate} are carried out
	by the station operator and the $u$-update minimizes a convex quadratic
	objective with linear constraints.   Both can be efficiently solved.

\item The $x$-update by the utility company needs $(u(n), \lambda(n))$
	from the station operator in iteration $n$.  In fact, from \eqref{eq:augmentedL.2},
	the station operator does not need to communicate the detailed station assignment
	$u(n) = (u_{aj}(n), \, a\in\mathbb A, \, j\in\mathbb N_w)$ to the utility company but only
	the total numbers of EVs $(u_j(n),j\in\mathbb{N}_w)$ that stations $j$ are assigned to.

\item	The $(u,\lambda)$-updates  by the station
	operator need in iteration $n$ the utility company's estimate $w(n+1)$ of
	$(r(M_j-m_j+u_j(n+1)), j\in\mathbb N_w)$.

\item The reason that the $x$-update by the utility company needs $(u_j(n),j\in\mathbb{N}_w)$ and
	the $u$-update by the station operator needs $w(n+1)$ is the (quadratic)
	regularization term in $h_\rho$.	  This becomes unnecessary for the dual
	decomposition approach in Sec.~\ref{subsec:dd} without the regularization
	term.
\eee
\end{remark}
	The communication structure is illustrated in Fig.~\ref{fig:ADMM}.
	In particular, private information of the utility company, such as distribution network
	parameters $z_{jk}$, network states $(s(n), s^g(n), v(n), l(n), S(n))$, cost functions $f$,
	and operational constraints, as well as private information of the station operator, such as the total number of batteries $(M_j,j\in\mathbb{N}_w)$, the numbers of available fully-charged batteries $(m_j,j\in\mathbb{N}_w)$, how many EVs or where they are or their states of
	charge, and the detailed assignment $u(n)$, do not need to be communicated.
\begin{figure}[htbp]
\centering
\includegraphics[width=0.25\textwidth]{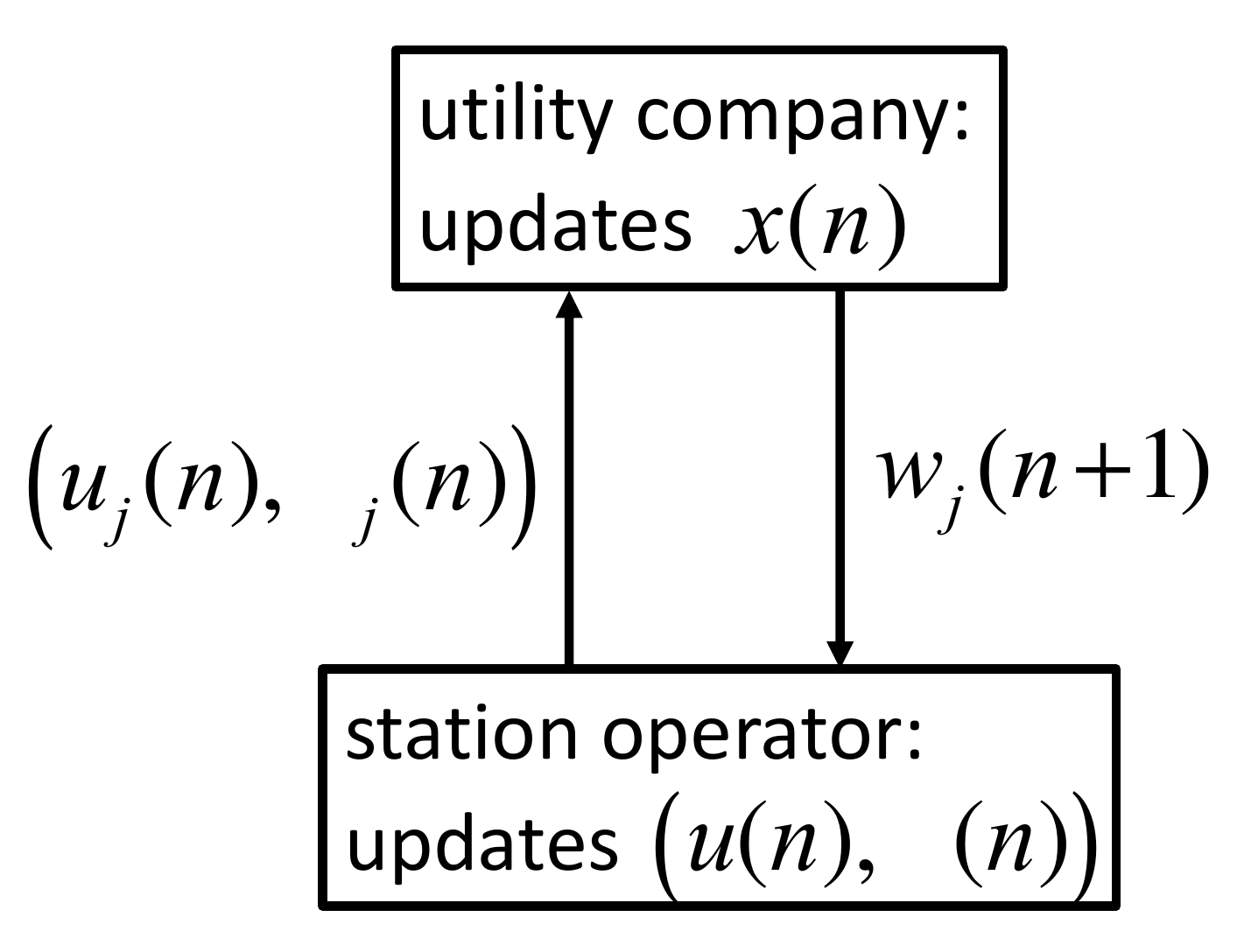}
\caption{Communication between utility company and station operator.}
\label{fig:ADMM}
\end{figure}

When the cost functions $f_j$ are closed, proper and convex and $L_0(x, u, \lambda)$ has
a saddle point, the ADMM iteration \eqref{eq:admmupdates} converges
in that, for all $j\in\mathbb N_w$, the mismatch $|w_j(n)-u_j(n)| \rightarrow 0$
and the objective $f(x(n))+g(u(n))$ converges to its minimum value \cite{boyd2011distributed}.
This does not automatically guarantee that $(x(n), u(n))$ converges to an optimal solution
to \eqref{eq:primal problem tr}.\footnote{In theory ADMM may converge and circulate around the set of optimal solutions, but never reach one. In practice a solution within a given error tolerance is acceptable.}   If $(x(n), u(n))$ indeed converges to a primal optimal solution
$(x^*, u^*)$, $u^*$ may generally not be binary.   We use a heuristic to derive a binary station assignment from $u^*$, as mentioned above.
Fortunately, the following result shows that the number of EVs $a$
for which a binary assignment needs to be derived from the obtained non-binary one
$(u^*_{aj}, j\in\mathbb N_w)$ is small.
\begin{theorem}\label{teo:number}
Suppose $(x^*, u^*)$ is an optimal solution to the relaxation \eqref{eq:primal problem tr},
then the number of EVs $a$ for which $u^*_{aj}<1$
for all $j\in\mathbb N_w$ is at most ${N_w(N_w-1)}/{2}$.
\end{theorem}

See Appendix \ref{sec:proof} for its proof. In practice, the number $N_w$ of stations is typically small compared with the number $A$ of EVs that request battery swapping. Therefore, it is expected that the discretization will attain a station assignment close to optimum, considering the subtle impact of individual EVs.


\subsection{Distributed solution via dual decomposition}\label{subsec:dd}

The ADMM-based algorithm assumes the station operator directly controls the station assignment to all EVs.
This requires that the station operator know the locations ($d_{aj}$), states of charge ($c_{a}$) and
performance ($\gamma_{a}$) of EVs.   Moreover, the aggregate EV information $u_j$ needs to
be provided to the utility company.
We now present another algorithm based on dual decomposition that is more suitable in
situations where it is undesirable or inconvenient to share private information
between the utility company, the station operator, and EVs.


In the original relaxation \eqref{eq:primal problem tr}, the update of the injections $p_j$ in \eqref{eq:inj}
by the utility company involves $u_j$ which are updated by the station operator.   These two computations
are decoupled in the ADMM-based solution by introducing the auxiliary variables $w_j$ at the utility
company and relaxing the constraint $w_j=r(M_j-m_j+u_j)$.
In addition, the station assignment $u$ must satisfy $u_j\leq m_j$ in
\eqref{eq:con.tax2.2}.  This
is enforced in the ADMM-based solution by the station operator that computes $u$ for the EVs.
To fully distribute the computation to individual EVs, we dualize $u_j\leq m_j$ as well.
Let $\lambda$ and $\mu\geq 0$ be the Lagrange multiplier vectors for the constraints $w_j=u_j$
and $u_j\leq m_j$, respectively.
Intuitively $w$ and $\lambda$ decouple the computation of the utility company and that
of individual EVs through coordination with the station operator. Additionally, $\mu$
decouples and coordinates the EVs' decisions so that EVs do not need
direct communication among themselves to ensure that their decisions $u_{aj}$ collectively satisfy
$u_j\leq m_j$.

Consider the Lagrangian of \eqref{eq:primal problem admm} with these two sets
of constraints relaxed:
\begin{small}
\beq
\begin{split}
L(x, {u},{\lambda}, \mu)  :=  f(x) + g(u) & +  \sum\limits_{j\in\mathbb{N}_w}\lambda_j(w_j-r(M_j-m_j+u_j))
\\
&  + \sum\limits_{j\in\mathbb{N}_w} \mu_j (u_j-m_j)
\label{eq:L}
\end{split}
\eeq\end{small}%
and the dual problem of \eqref{eq:primal problem admm}:
\bqn
\max_{\lambda, \mu\geq 0} & & D(\lambda, \mu) \ := \ \min_{x\in\mathbb X,u\in\hat{\mathbb U}}\ L(x, u, \lambda, \mu)
\eqn
where the constraint set $\hat{\mathbb U} $ on $u$ is:
\bqn
\hat{\mathbb U} & := & \{ u\in \mathbb{R}^{A N_w} \, :\,
	u \text{ satisfies } \eqref{eq:con.tax3.2}\eqref{eq:con.tax1.2}  \}
\eqn
Let $u_a := (u_{aj}, j\in\mathbb N_w)$ denote the column vector
of EV $a$'s decision on which station to swap its battery.   Then the dual problem is
separable in power flow variables $x$ as well as individual EV decisions
$u_a$:
\begin{subequations}
\bq
D(\lambda, \mu) & =: & V(\lambda) \ + \ \sum_{a\in\mathbb A} U_a(\lambda, \mu)
\label{eq:dualF.1}
\eq
where the problem $V(\lambda)$ to be solved by the utility company is:
\bq
V(\lambda) & := &
\min_{x\in\mathbb X} \left( f(x) + \sum_{j\in\mathbb N_w} \lambda_j w_j \right)
\label{eq:dualF.2}
\eq
and  the problem $U_a(\lambda)$ to be solved by each individual EV $a$ is:
\bq
\!\!\!\!\!\!\!\!\!
U_a(\lambda, \mu) & := & \min_{u_a\in \hat{\mathbb U}_a}
	\sum_{j\in\mathbb N_w} \left( \alpha d_{aj} - r \lambda_j + \mu_j \right) u_{aj}
\label{eq:dualF.3}
\eq
where the constraint set $\hat{\mathbb U}_a$ on $u_a$ is:
\beq\nonumber
\ \ \ \ \hat{\mathbb U}_a \ := \ \Big\{ u_a\in \mathbb{R}^{N_w} \, : \!\!\!\!\!\!\!\!\!\!\!\!\!\!\!\!\!\!\!\!\!\!\!\!\!\!\!\!\!\!\!\!
\begin{split}
&  u_{aj}\in\left[0,1\right], \, j\in\mathbb{N}_w \\
&  u_{aj}=0 \  \mathrm{if}~d_{aj} > \gamma_a c_a, \ j\in\mathbb{N}_w \\
& \sum\limits_{j\in\mathbb{N}_w} {u_{aj}}=1
\end{split} \ \Big\}
\eeq
Note that \eqref{eq:dualF.3} entails closed-form solutions. If there exists a unique optimal solution to $U_a(\lambda, \mu)$, i.e., for any EV $a$ there is a unique $j_a^*(\lambda,\mu)$ defined as
\bqn
j_a^*(\lambda, \mu) & := & \arg \min_{j: d_{aj} \le \gamma_a c_a} \{\alpha d_{aj}-r\lambda_j+\mu_j\}
\eqn
then the optimal solution to $U_a(\lambda, \mu)$ can be uniquely determined as
\bqn
u_{aj}(n) & := & \left\{ \begin{array}{lcl}
		1 & \text{if} & j = j_a^*(\lambda, \mu)  \\
		0 & \text{if} & j\neq j_a^*(\lambda, \mu)
		\end{array}  \right.
\eqn
\label{eq:dualF}%
\end{subequations}
i.e., it simply chooses
the unique station $j_a^*$ within EV $a$'s driving range that has the minimum cost
$\alpha d_{aj}-r\lambda_j+\mu_j$.

From \eqref{eq:L} the standard dual algorithm for solving \eqref{eq:primal problem admm}
is: for $j\in\mathbb N_w$,
\begin{subequations}
\begin{small}
\bq
\!\!\! \!\!\! \!\!\!\!\!\!\!\!\!\!\!\!\!\!\!\!\!\!\!\! && \lambda_j(n+1)   :=   \lambda_j(n) \, + \, \rho_1(n) [w_j(n) - r(M_j-m_j+u_j(n))]
\label{eq:DualAlg.1}\\
\!\!\! \!\!\! \!\!\!\!\!\!\!\!\!\!\!\!\!\!\!\!\!\!\!\! && \mu_j(n+1)   :=  \max\{\, \mu_j(n) \, + \, \rho_2(n) (u_j(n) - m_j), 0 \, \}
\label{eq:DualAlg.2}
\eq\end{small}%
where $\rho_1(n), \rho_2(n)>0$ are diminishing stepsizes and,
from \eqref{eq:dualF}, we have
\begin{small}
\bq
\!\!\!\!\!\!
x(n) & \!\!\!\!  := \!\!\!\!  & \arg\min_{x\in\mathbb X}\ \left( f(x) + \sum_{j\in\mathbb N_w} \lambda_j(n) w_j \right)
\label{eq:DualAlg.3}
\\
\!\!\!\!\!\!
u_{a}(n) & \!\!\!\! := \!\!\!\!  & \arg\min_{u_a \in\hat{\mathbb U}_a} \sum_{j\in\mathbb N_w} \left( \alpha d_{aj} - r \lambda_j + \mu_j \right) u_{aj},  \, a\in\mathbb A
\label{eq:DualAlg.4}
\eq\end{small}%
\label{eq:DualAlg}%
\end{subequations}

\begin{remark}
\bee
\item The $x$-update \eqref{eq:DualAlg.3} is carried out by the utility company and
	involves minimizing a convex objective with convex quadratic constraints.
	The only information that is non-local to the utility company for its
	$x$-update is one of the dual variables $\lambda(n)$ computed by the
	station operator.

\item The $u_a$-update \eqref{eq:DualAlg.4} is carried out by each individual EV.
	Each EV requires the dual variables $(\lambda(n), \mu(n))$ from the
	station operator for its update.
	
\item	The dual updates \eqref{eq:DualAlg.1}\eqref{eq:DualAlg.2} are carried out by the
	station operator which uses a (sub)gradient ascent algorithm to solve the dual
	problem $\max_{\lambda, \mu\geq 0} D(\lambda, \mu)$.  It requires $w(n)$
	from the utility company and individual decisions $u_a(n)$ from EVs $a$.
\eee
\end{remark}
The communication structure is illustrated in Fig. \ref{fig:dual}.
\begin{figure}[htbp]
\centering
\includegraphics[width=0.4\textwidth]{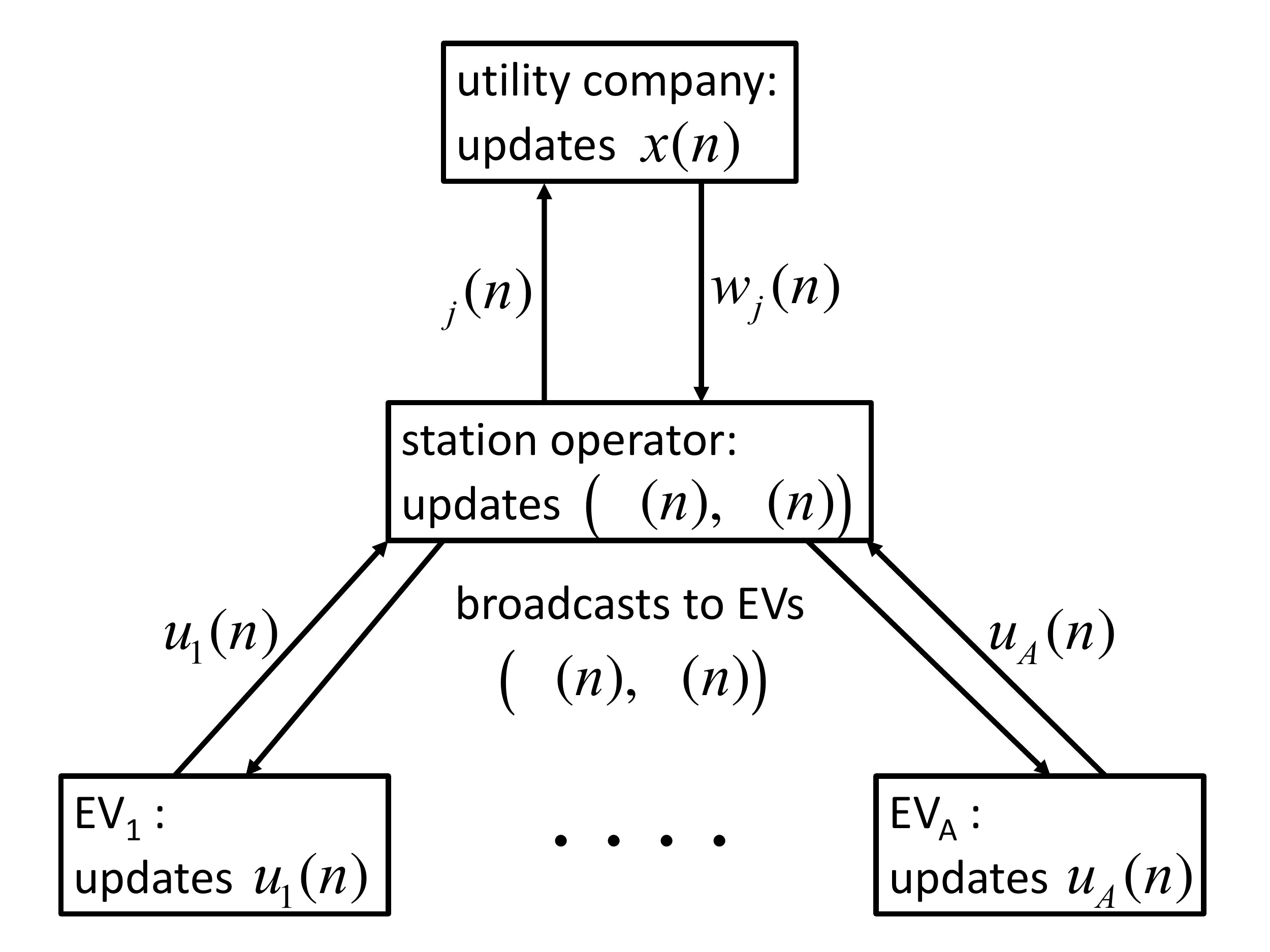}
\caption{Communication between utility company, station operator, and EVs.}
\label{fig:dual}
\end{figure}
In particular, EVs are completely decoupled from the utility company and
among themselves.
Unlike the ADMM-based approach, the station operator knows only the (tentative)
battery swapping decisions of EVs, but not their private information such as their locations
($d_{aj}$), states of charge ($c_{a}$) or performance ($\gamma_{a}$).

Since the relaxation \eqref{eq:primal problem tr} is convex, strong duality holds if Slater's constraint qualification is satisfied. By strong duality, when the above (sub)gradient algorithm converges to a dual optimal solution $(\lambda^*,\mu^*)$, any primal optimal point is also a solution to the corresponding $x$-update \eqref{eq:DualAlg.3} and $u_a$-update \eqref{eq:DualAlg.4} \cite{boyd2003subgradient, boyd2004convex}. Suppose $(x(n),u_a(n),a\in\mathbb A)$ indeed converges to a primal optimal solution $(x^*,u_a^*,a\in\mathbb A)$, typically $(u_a^*,a\in\mathbb A)$ may not be binary. As discussed above, the discretization of $(u_a^*,a\in\mathbb A)$ can be readily implemented since \emph{\textbf{Theorem}~\ref{teo:number}} still holds, and in practice the gap to optimum is small.

\begin{remark}
The two solutions have their own advantages and can be adapted to different application scenarios. Specifically, the ADMM-based one requires a station operator that is trustworthy and can access EVs' private information. Nonetheless, thanks to the station operator that is able to optimize the station assignment on behalf of all EVs, no computation unit is necessary on each EV's end, and communications are only required between the station operator and the utility company, which is practical to realize. In contrast, the dual-decomposition-based one is more distributed in a sense, thus further preserving privacy inherently. However, it necessitates computation capabilities of all EVs. In addition, iterative communications, both between the station operator and the utility company and between the station operator and each individual EV, are required to enable computation. As a result, the deployment of computation units on each EV and communication unreliability may impede the practical implementation of the dual-decomposition-based solution. To conclude, the dual-decomposition-based solution further preserves privacy at the price of communication and computation overheads, compared with the ADMM-based one.
\end{remark}

\section{Numerical results}\label{sec:numerical}

We test the two distributed algorithms on the same 56-bus radial distribution feeder of Southern California Edison (SCE) in Part I. Details about the feeder can be found in \cite{farivar2012optimal}. Similar setups are adopted to demonstrate the algorithm performance.
4 distributed generators and 4 stations are added to the feeder at different buses, and the 4 stations are assumed to be evenly distributed in a 4km$\times$4km square area supplied by the feeder. Table \ref{tab:setup} lists their parameters.\footnote{The units of the real power, reactive power, cost (for the whole control interval), distance and weight in this paper are MW, Mvar, \$, km and \$/km, respectively.} We use examples of $A=400$ EVs that request battery swapping in a certain control interval. We randomize their current locations uniformly within the square area and ignore their destinations. $d_{aj}$ is assumed to be the Euclidean distance, and the driving range constraints are readily satisfied by assuming all EVs can reach any of the 4 stations for illustrative purposes. The constant charging rate is $r=0.01$MW \cite{yilmaz2013review}, and the weight is $\alpha=0.02$\$/km \cite{eia2011annual}. Simulations are run on a laptop with Intel Core i7-3632QM CPU@2.20GHz, 8GB RAM, and 64-bit Windows 10 OS.

\begin{table}[htbp]
  \caption{Setup}\label{tab:setup}
  \begin{center}
  \subtable[Distributed generator]{
  \begin{tabular}{ccccccc}
  \toprule
     Bus & $\overline{p}_j^g $ & $\underline{p}_j^g $ & $\overline{q}_j^g $  & $\underline{q}_j^g $  & Cost function \\
  \midrule
  \rowcolor[gray]{.9}   1 & 4 & 0 & 2 & -2 & $0.3 {p^g}^2 + 30 p^g$ \\
    4 & 2.5 & 0 & 1.5 & -1.5 & $0.1 {p^g}^2 + 20 p^g$ \\
  \rowcolor[gray]{.9}   26 & 2.5 & 0 & 1.5 & -1.5 & $0.1 {p^g}^2 + 20 p^g$ \\
    34 & 2.5 & 0 & 1.5 & -1.5 & $0.1 {p^g}^2 + 20 p^g$ \\
  \bottomrule
  \end{tabular}\label{tab:dgsetup}
  }
  \subtable[Station]{
  \begin{tabular}{ccccccc}
  \toprule
     Bus &  Location & $M_j$ & $m_j$   \\
  \midrule
  \rowcolor[gray]{.9}   5 & (1,1) & $m_j$ &  (i) $A$ \ ; \ (ii) $A/2$  \\
    16 & (3,1) & $m_j$ & \ \ (i) $A$ \ ; \ (ii) $A/10$ \\
  \rowcolor[gray]{.9}  31 & (1,3) & $m_j$ &  (i) $A$ \ ; \ (ii) $A/4$  \\
    43 & (3,3) & $m_j$ &(i) $A$ \ ; \ (ii) $A/4$  \\
  \bottomrule
  \end{tabular}\label{tab:stationsetup}
  }
  \end{center}
\end{table}

As shown in Table \ref{tab:stationsetup}, we test the two distributed algorithms using two cases (i) and (ii) of different $m_j$'s to show their convergence, the suboptimality they can achieve, and the exactness of SOCP relaxation.

\vspace{0.05in}
\noindent
\textbf{Convergence}: The convergence of the ADMM-based algorithm in case (i) is demonstrated in Fig.~\ref{fig:ADMMfig}, where Fig.~\ref{fig:ADMMconv} illustrates that the Lagrange multiplier vector $\lambda$ converges rapidly and Fig.~\ref{fig:ADMMresidues} shows the residual of the relaxed equality constraint \eqref{eq:PPA.3} diminishes acoordingly so as to attain consensus of $w$ and $(u_j,j\in\mathbb{N}_w)$ between the utility company and the station operator. Similar results can be found for case (ii). In terms of the dual-decomposition-based algorithm, Fig.~\ref{fig:ddlambda} and Fig.~\ref{fig:ddmu} show the convergence of its two Lagrange multiplier vectors $\lambda$ and $\mu$ respectively in case (ii). $\lambda$ maintains the consensus between the utility company and EVs at convergence, and $\mu$ guarantees \eqref{eq:con.tax2.2} is satisfied when it converges. Dual decomposition usually takes more iterations to converge due to extra iterative coordination among all EVs by updating $\mu$. For case (i), results are similar except that $\mu$ remains 0 during iterations as \eqref{eq:con.tax2.2} is always satisfied.

\begin{figure}
\centering
\subfigure[]{
	\includegraphics[width=0.35\textwidth]{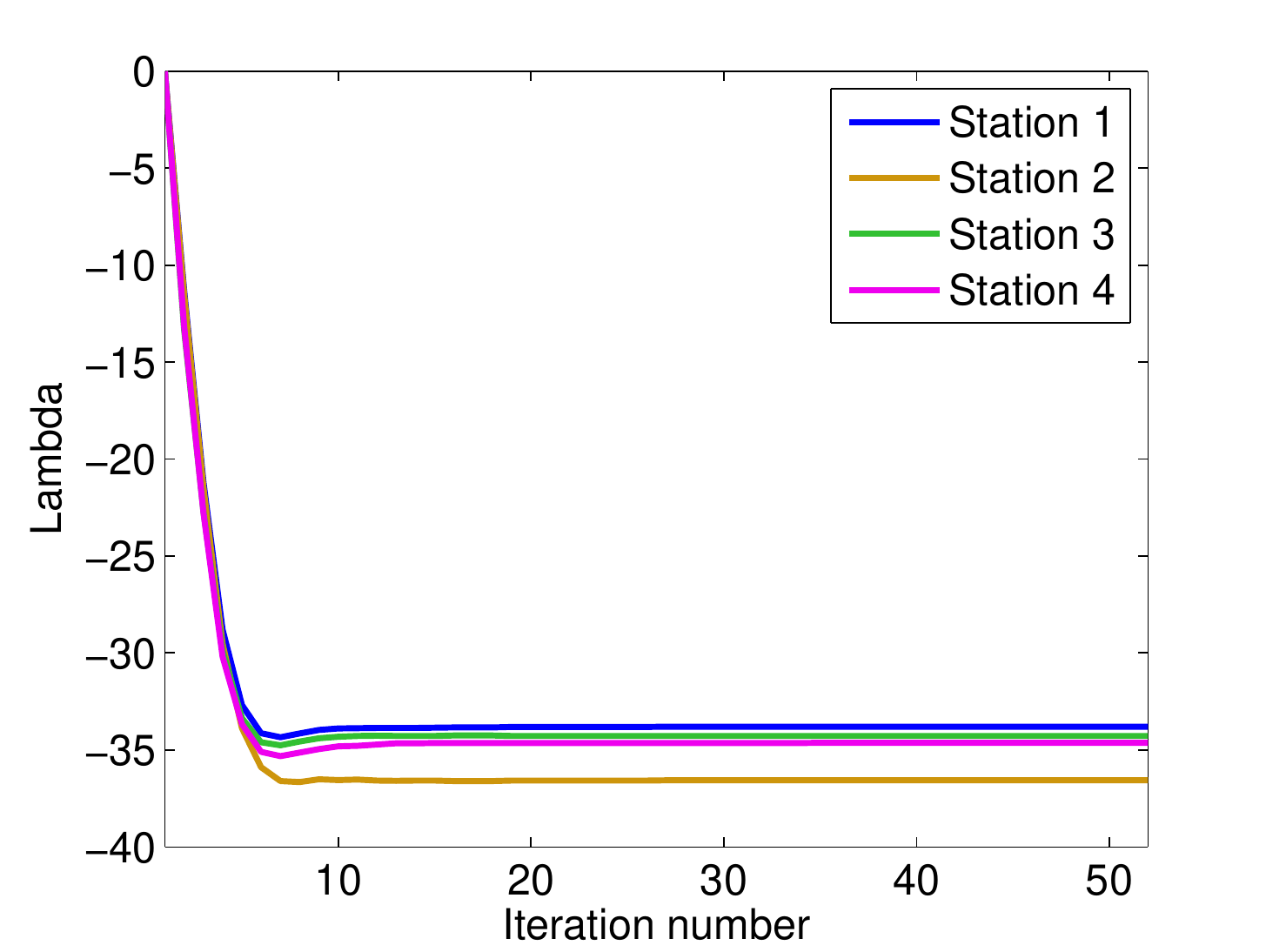}
	\label{fig:ADMMconv}
}
\hspace{-0.22in}
\subfigure[]{
	\includegraphics[width=0.35\textwidth]{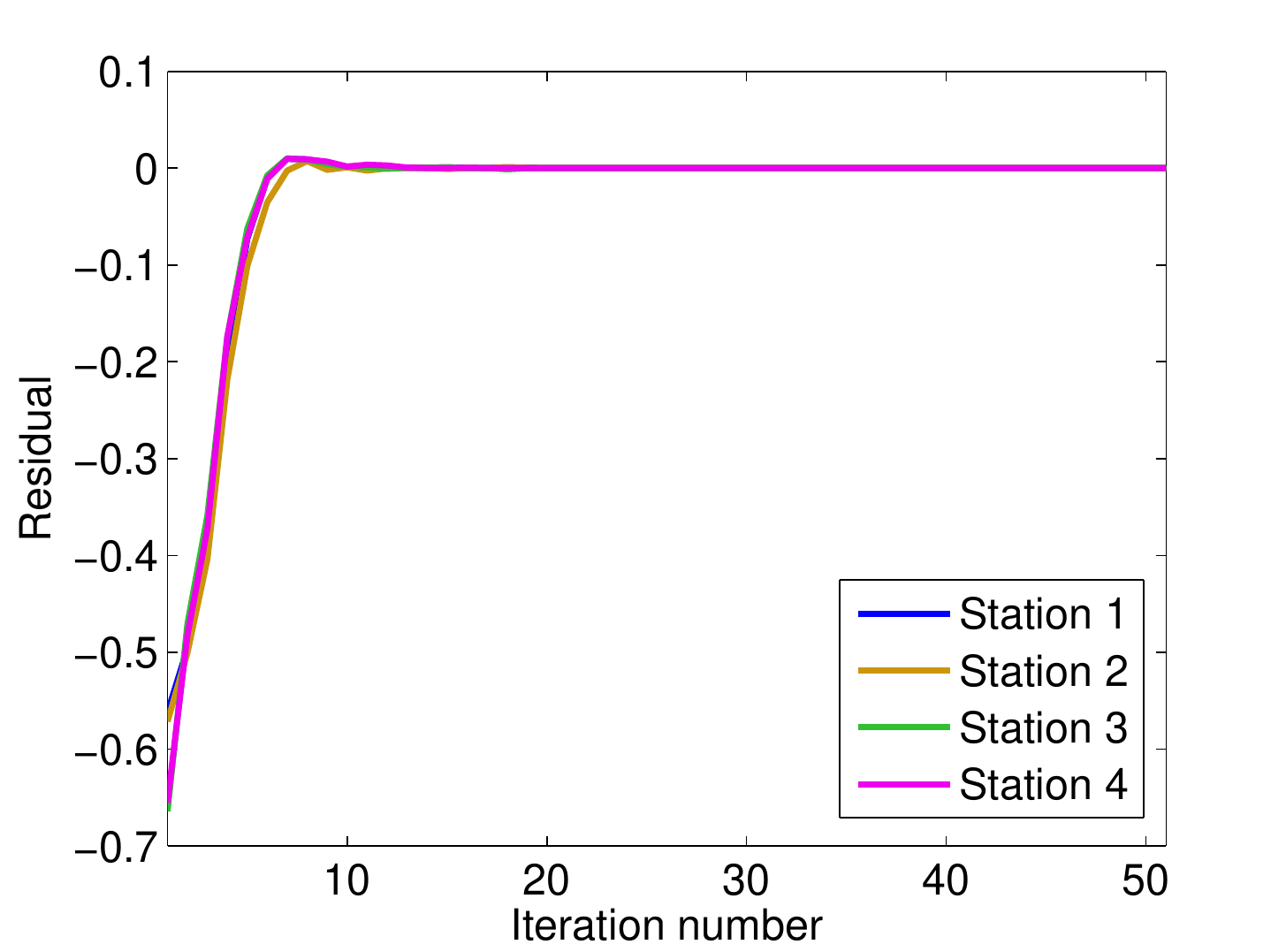}
	\label{fig:ADMMresidues}
}
\caption{Convergence of ADMM (a) $\lambda$. (b) Residual.}
\label{fig:ADMMfig}
\end{figure}

\begin{figure}
\centering
\subfigure[]{
	\includegraphics[width=0.35\textwidth]{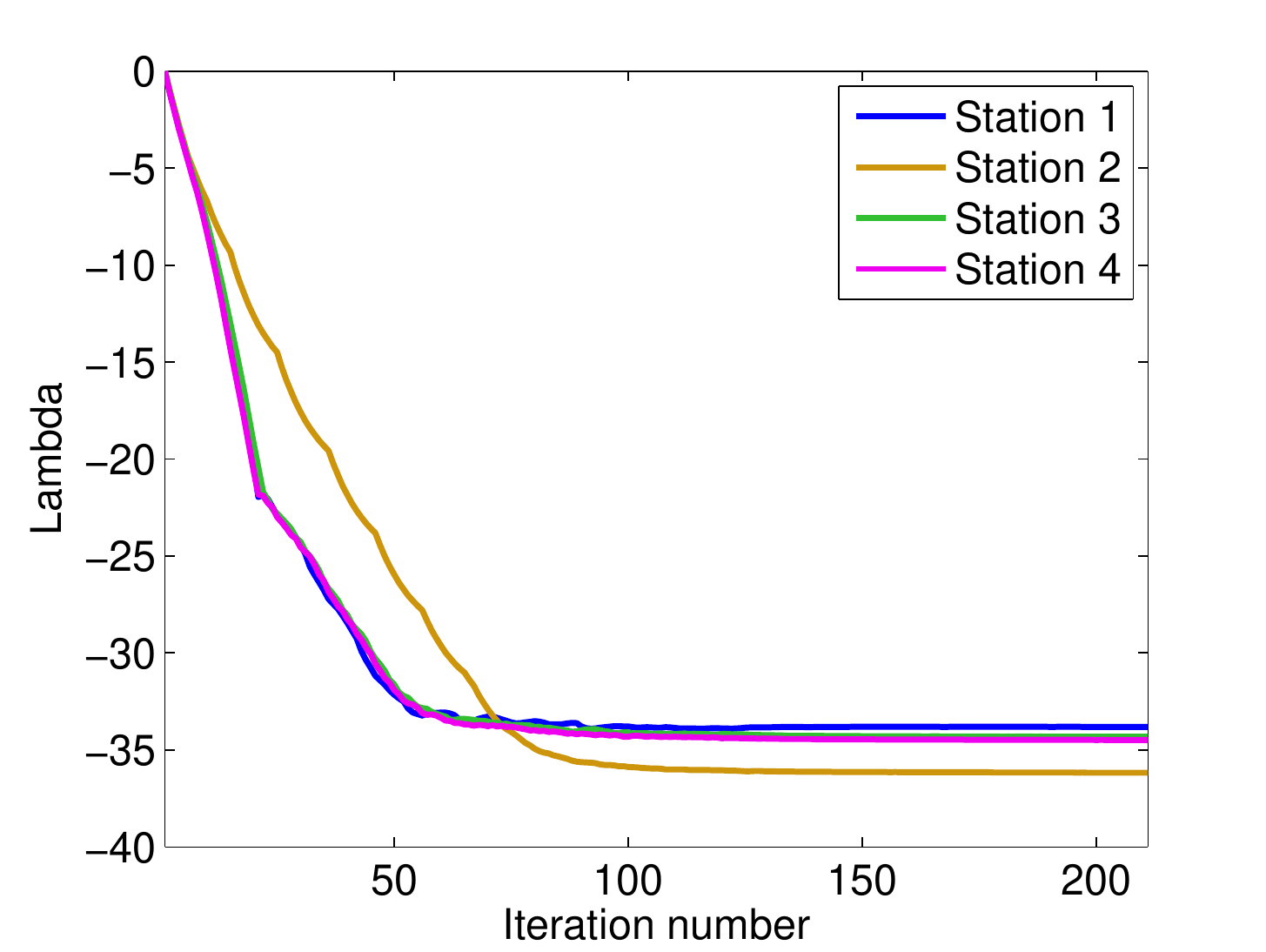}
	\label{fig:ddlambda}
}
\hspace{-0.22in}
\subfigure[]{
	\includegraphics[width=0.35\textwidth]{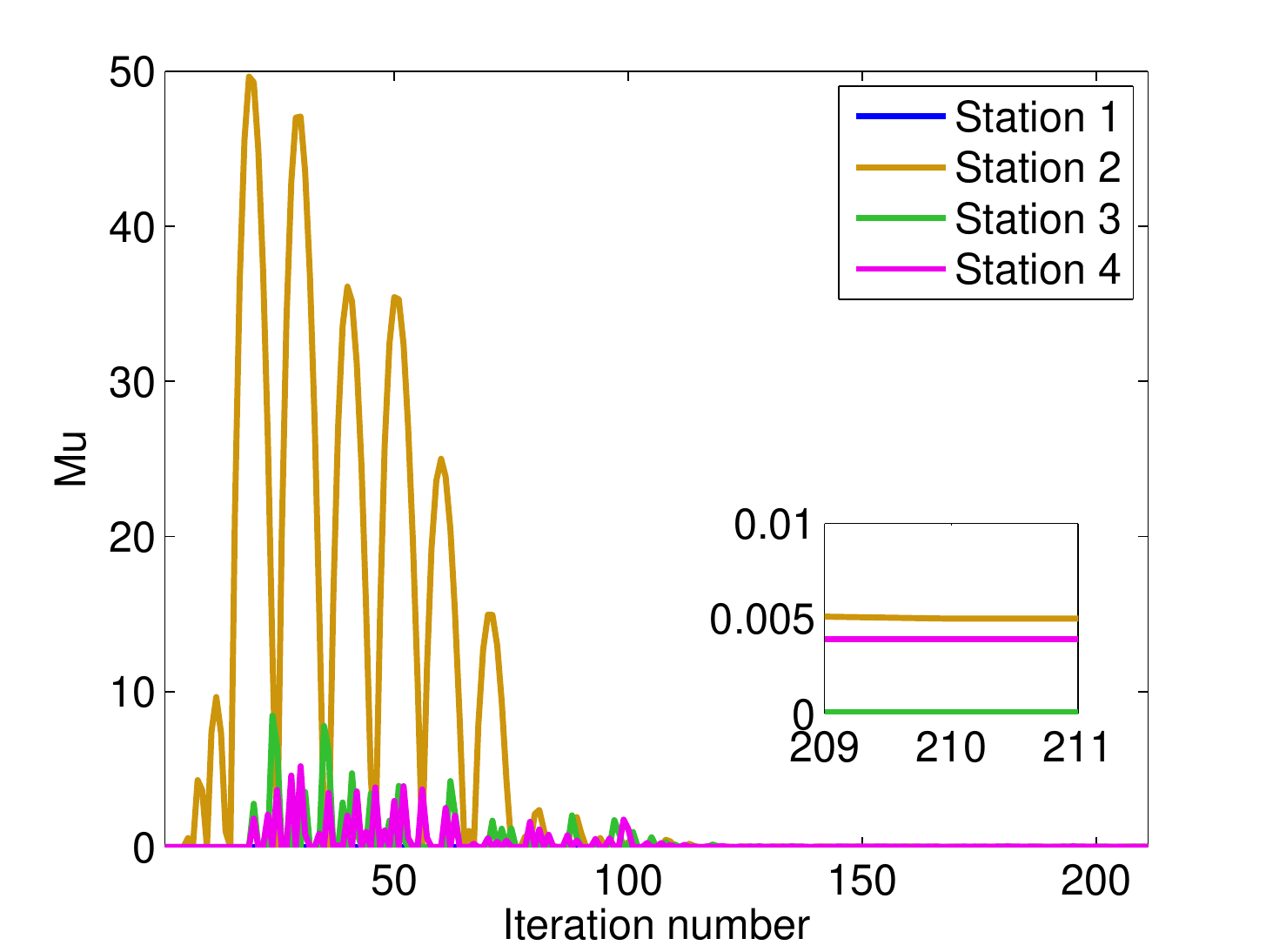}
	\label{fig:ddmu}
}
\caption{Convergence of dual decomposition (a) $\lambda$. (b) $\mu$.}
\label{fig:ddfig}
\end{figure}

\vspace{0.05in}
\noindent
\textbf{Suboptimality (comparison with centralized solution)}: In case (i) both algorithms obtain a solution in which the station assignment to two EVs, marked black in Fig.~\ref{fig:casea}, is non-binary, i.e., $u_{242}=[0.707~ 0.293~ 0.000~ 0.000]$ and $u_{367}=[0.230~ 0.000~ 0.770~ 0.000]$. This is consistent with \emph{\textbf{Theorem}~\ref{teo:number}}. Unlike the centralized solution in Part I that uses global information to compute a globally optimal binary assignment, the distributed algorithms solve the relaxation \eqref{eq:primal problem tr} and are likely to attain a non-binary one, exemplified by case (i). However, we can naively round $u_{243}$ and $u_{367}$ off to discretize the assignment, as shown in Fig.~\ref{fig:casea}, which turns out to be an optimal station assignment of the original problem \eqref{eq:primal problem}, identical with the one computed by the centralized solution in Part I. Typically a heuristic binarization method will yield a suboptimal assignment, but the suboptimality gap is usually trivial due to the bounded number of EVs with a non-binary assignment and their limited impact on the whole system.

In case (ii) we reduce available fully-charged batteries at each station to activate \eqref{eq:con.tax2.2}. Fig.~\ref{fig:caseb} shows the solution achieved by both algorithms which contains a binary assignment, also an optimal station assignment of the original problem \eqref{eq:primal problem}. In this case the relaxation of binary variables is exact. EVs, to which the station assignment is altered due to the bound imposed on battery availability of each station, are marked cyan in Fig.~\ref{fig:caseb}. The intuition is that an active \eqref{eq:con.tax2.2} sometimes can help eliminate the marginal non-binary assignment to certain EVs, and this is often the case in practice where the battery availability is uneven across stations.

\begin{figure}
\centering
\subfigure[]{
	\includegraphics[width=0.35\textwidth]{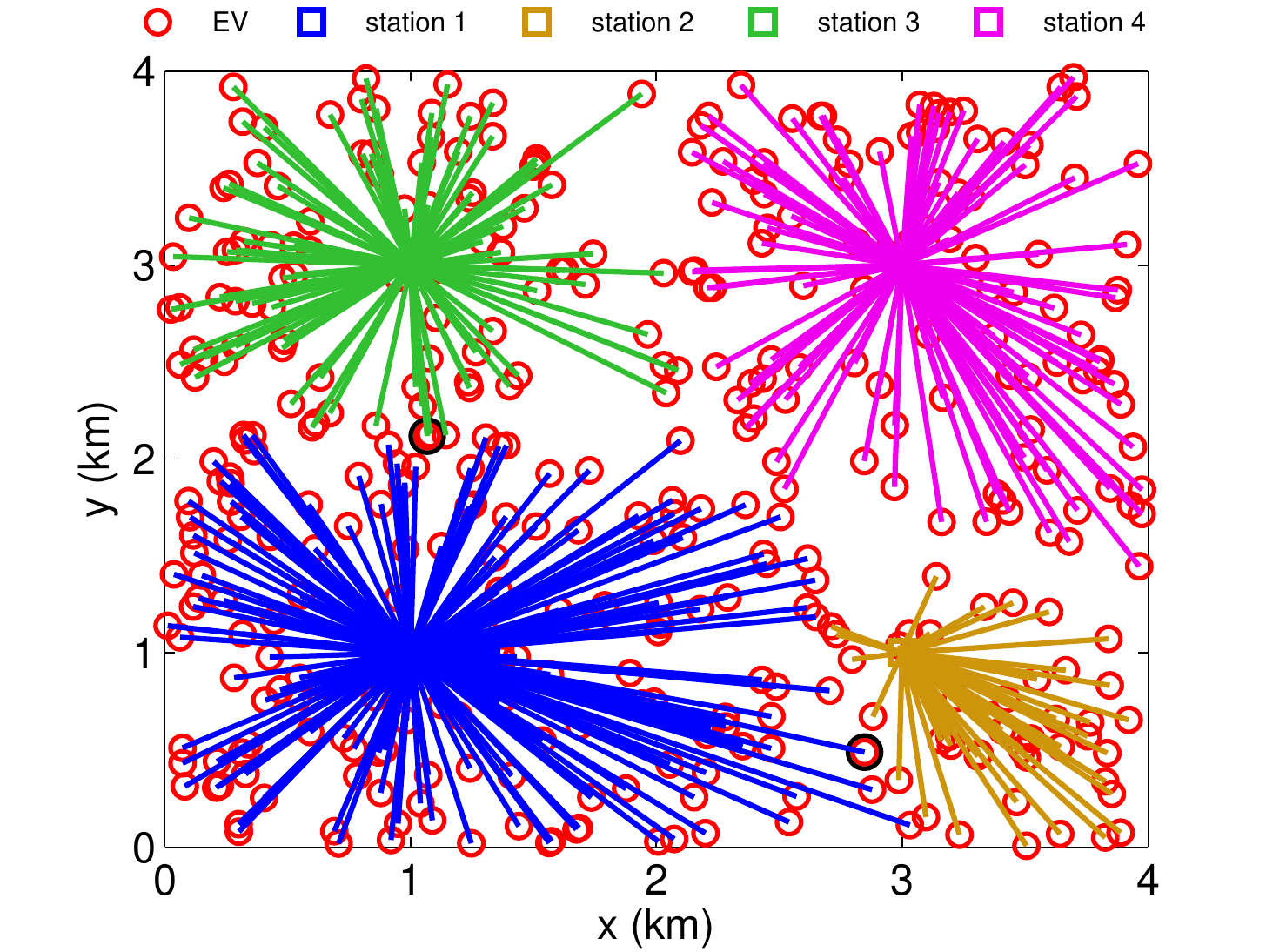}
	\label{fig:casea}
}
\hspace{-0.22in}
\subfigure[]{
	\includegraphics[width=0.35\textwidth]{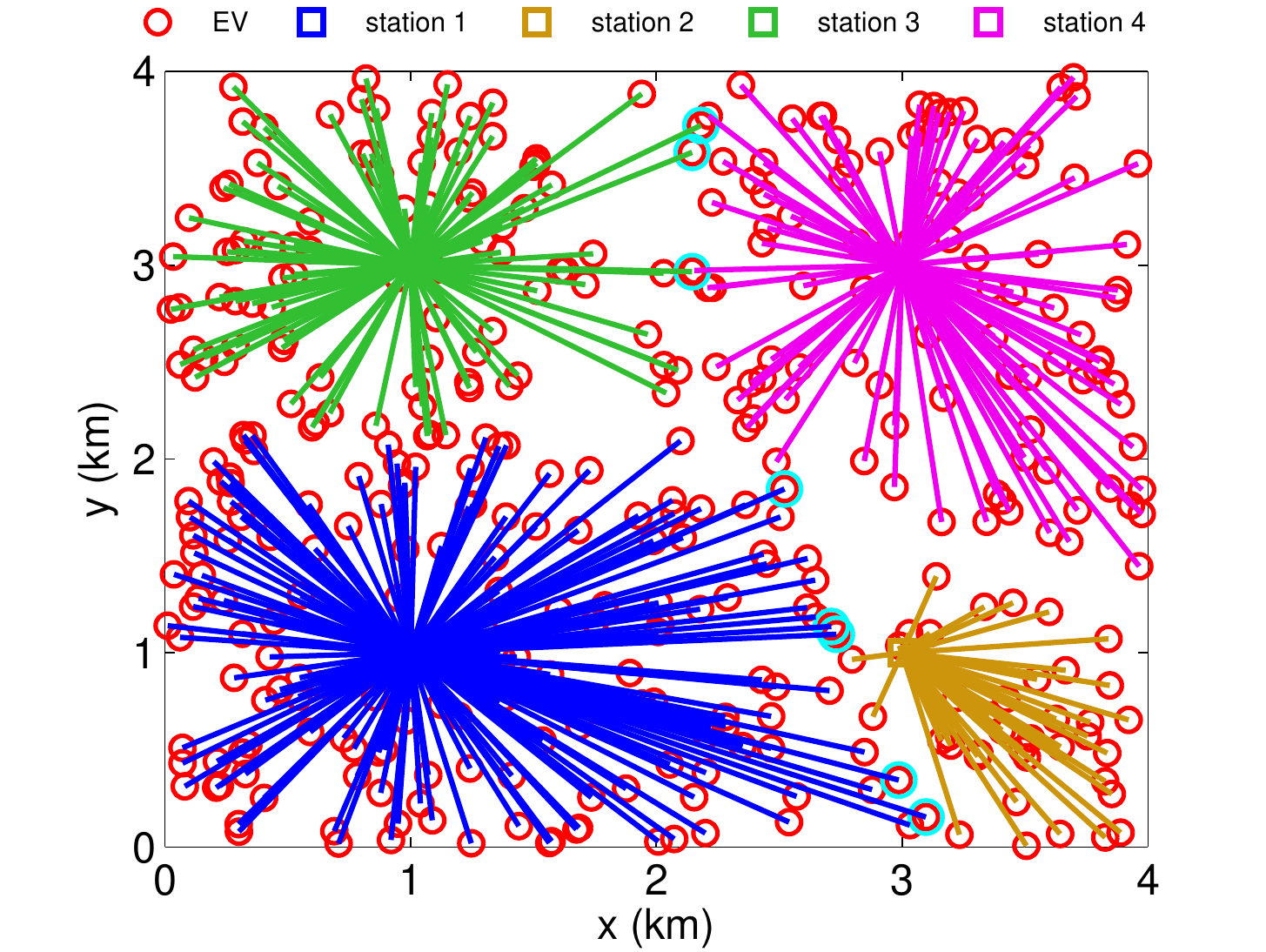}
	\label{fig:caseb}
}
\caption{Suboptimality in different cases (a) case (i). (b) case (ii).}
\label{fig:suboptimality}
\end{figure}

\vspace{0.05in}
\noindent
\textbf{Exactness of SOCP relaxation}: We check whether the above solutions computed by the distributed algorithms attain equality in \eqref{eq:df2.c}, i.e., whether the SOCP relaxation of OPF is exact. Our final results, as well as most tests in other case studies, confirm the exactness of the SOCP relaxation. Only partial result data for case (ii) are listed in Table \ref{tab:exactness} due to space limit. To sum up, the above solutions satisfy power flow equations and are thus implementable.

\begin{table}[htbp]
  \caption{Exactness of SOCP relaxation (partial results for case (ii))}\label{tab:exactness}
  \begin{center}
  \begin{tabular}{ccccc}
  \toprule
  \multicolumn{2}{c}{Bus} & \multirow{2}{*}{ $v_jl_{jk} $} & \multirow{2}{*}{ $|S_{jk}|^2$} & \multirow{2}{*}{Residual}  \\
   From & To & & & \\
  \midrule
  \rowcolor[gray]{.9}   1 & 2 &  2.582 & 2.582 & 0.000 \\
    2 & 3 &  0.006 &  0.006 & 0.000 \\
  \rowcolor[gray]{.9}  2 & 4 & 2.336 & 2.336 & 0.000 \\
    4 & 5 & 3.413 & 3.413 & 0.000 \\
  \rowcolor[gray]{.9}  4 & 6 &  0.005 &  0.005 & 0.000 \\
      4 & 7 & 2.276 & 2.276 & 0.000 \\
  \rowcolor[gray]{.9}  7 & 8 &  1.984 &  1.984 & 0.000 \\
      8 & 9 &  0.009 &  0.009  & 0.000 \\
  \rowcolor[gray]{.9}  8 & 10 & 1.518 & 1.518 & 0.000 \\
        10 & 11 & 1.318 & 1.318 & 0.000 \\
  \bottomrule
  \end{tabular}
  \end{center}
\end{table}

\section{Concluding remarks}\label{sec:conclusion}

This paper is an extension of Part I that solves the same optimal scheduling problem for battery swapping. Instead of resorting to a centralized solution which requires global information, two distributed solutions based on ADMM and dual decomposition respectively are proposed, considering the fact that the grid, battery stations and EVs are likely operated by different entities, which do not share their private information. The distributed algorithms allow these entities to make individual decisions but coordinate through privacy-preserving information exchanges to jointly solve an approximate version of the joint battery swapping scheduling and OPF problem. Finally, the obtained station assignment may not be binary, but we prove that the number of EVs for which a binary assignment needs to be derived from the obtained non-binary one is small, and thus a station assignment that is close to optimum is expected to be attained. Numerical tests on the SCE 56-bus distribution feeder demonstrate the algorithm performance and validate our analysis.

\bibliographystyle{ieeetr}
\bibliography{bib2}

\appendices
\section{Proof of Theorem \ref{teo:number}}
\label{sec:proof}

We refer to EV $a$ as a \emph{critical} EV if $u_{aj}^*<1$ for all $j\in\mathbb{N}_w$. We first show the following lemma always holds, on which basis \emph{\textbf{Theorem}~\ref{teo:number}} is then proved.

\begin{lemma}\label{lem:critical}
There exists an optimal solution with no $u_{aj},u_{ak},u_{bj},u_{bk}>0$ for $\forall a,b\in\mathbb{A}$ and $\forall j,k\in\mathbb{N}_w$, i.e., if a solution contains two critical EVs $a,b\in\mathbb{A}$ with $u_{aj},u_{ak},u_{bj},u_{bk}>0$ for certain $j,k\in\mathbb{N}_w$, there always exists a better one.
\end{lemma}

\begin{proof}
We prove \emph{\textbf{Lemma}~\ref{lem:critical}} by contradiction.

Suppose \eqref{eq:primal problem tr} is solved by the proposed algorithms and we obtain the primal optimal solution $({u}^*,{x}^*)$, which contains the assignment of two stations to two same critical EVs, i.e., $u_{aj}^*,u_{ak}^*,u_{bj}^*,u_{bk}^*>0$. We focus on the two-EV-two-station subsystem and fix its impact on the remaining parts of the whole system, e.g., the distribution network, other stations and EVs. Note that EV $a$ and EV $b$ have $r B_a$ and $r B_b$ charging loads to distribute, respectively, where $B_a:=u_{aj}^*+u_{ak}^*$ and $B_b:=u_{bj}^*+u_{bk}^*$. Meanwhile, the two EVs yield in total $r B_j$ and $r B_k$ charging loads at the buses of station $j$ and $k$, respectively, where $B_j:=u_{aj}^*+u_{bj}^*$ and $B_k:=u_{ak}^*+u_{bk}^*$. Apparently, $B_a+B_b=B_j+B_k$. Since $a$ and $b$, as well as $j$ and $k$, are indiscriminative, without loss of generality we assume case 1: $B_a\ge B_j \ge B_k \ge B_b$ and case 2: $ B_j \ge B_a \ge B_b \ge B_k $ to cover all possibilities. Below we will take case 1 as an illustrative example to go on with the proof, and case 2 shares the similar property.

Obviously the above circumstance $u_{aj}^*,u_{ak}^*,u_{bj}^*,u_{bk}^*>0$ would only occur when subcase 1': $d_{aj}\le d_{bj},d_{ak}\le d_{bk}$ or subcase 2': $d_{bj}\le d_{aj},d_{bk}\le d_{ak}$. Otherwise, EV $a$ and EV $b$ would have a bias for different stations in terms of the travel distance. For example, if $d_{aj}\le d_{bj},d_{bk}\le d_{ak}$, the two-EV-two-station subsystem will benefit if EV $a$ goes to station $j$ and EV $b$ goes to station $k$ with priority, i.e., $u'_{aj}=B_j,u'_{ak}=B_k-B_b,u'_{bj}=0,u'_{bk}=B_b$ with the rest of the optimal solution $({u}^*,{x}^*)$ is a better solution. Because the remaining parts of the whole system is not affected, which means the OPF solution and other EVs' total travel distance won't change, but the travel distance of the two-EV-two-station subsystem will decrease.

Again we take subcase 1' as an example, which we call case 11'. In this case, we may have
\beq\label{eq:case1}
\begin{split}
& d_{bj}- d_{bk}> d_{aj}-d_{ak}\\
\Longleftrightarrow ~& u_{aj}^*d_{aj}+u_{ak}^*d_{ak}+u_{bj}^*d_{bj}+u_{bk}^*d_{bk} \\
& >  B_j d_{aj} + (B_k-B_b) d_{ak} + B_b d_{bk}
\end{split}
\eeq
or
\beq\label{eq:case2}
\begin{split}
& d_{aj}-d_{ak}> d_{bj}- d_{bk}\\
\Longleftrightarrow ~& u_{aj}^*d_{aj}+u_{ak}^*d_{ak}+u_{bj}^*d_{bj}+u_{bk}^*d_{bk} \\
& > (B_j-B_b )d_{aj} + B_k d_{ak} + B_b d_{bj}
\end{split}
\eeq
or
\beq\label{eq:case3}
\begin{split}
& d_{aj}-d_{ak}= d_{bj}- d_{bk}\\
\Longleftrightarrow ~& u_{aj}^*d_{aj}+u_{ak}^*d_{ak}+u_{bj}^*d_{bj}+u_{bk}^*d_{bk} \\
& =  B_j d_{aj} + (B_k-B_b) d_{ak} + B_b d_{bk}\\
& = (B_j-B_b )d_{aj} + B_k d_{ak} + B_b d_{bj}
\end{split}
\eeq
In the cases of \eqref{eq:case1} and \eqref{eq:case2}, there is a better solution that decreases the travel distance of the two-EV-two-station subsystem, while the remaining parts of the whole system remain the same. This conflicts the original assumption that $({u}^*,{x}^*)$ is the optimal solution. In the case of \eqref{eq:case3}, still we can find solutions of equal optimum that only render either EV $a$ or EV $b$ critical.

Likewise, case 12', case 21' and case 22' share the same conclusion. As a result, there always exists an optimal solution with no $u_{aj},u_{ak},u_{bj},u_{bk}>0$ for $\forall a,b\in\mathbb{A}$ and $\forall j,k\in\mathbb{N}_w$.



This completes the proof.
\end{proof}

\vspace{0.1in}
Based on \emph{\textbf{Lemma}~\ref{lem:critical}}, we now go on to prove \emph{\textbf{Theorem}~\ref{teo:number}}.
\begin{proof}
According to the definition of a critical EV, its charging load will split into at least two parts that are distributed to different stations. The problem is currently transformed to how many critical EVs we can assign the $N_w$ stations to at most without violating \emph{\textbf{Lemma}~\ref{lem:critical}}.

This is a basic assignment problem. Obviously, assume every critical EV only splits into two parts at best, i.e., at most two station can be assigned to each critical EV. Then we need to assign the $N_w$ stations to as many critical EVs as possible without repeats. We start from assigning two consecutively indexed stations to each critical EV, i.e., stations 1 and 2 to one critical EV, and stations 2 and 3 to another, etc. Note that stations $N_w$ and 1 are not consecutive. By this means, we can assign stations to at most $N_w-1$ critical EVs. Then two stations with a one-index gap are assigned to each critical EV, i.e., stations 1 and 3 to one critical EV, and stations 2 and 4 to another, etc, by which means we can assign stations to at most $N_w-2$ critical EVs. By analogy, we finally assign two stations with an $(N_w-2)$-index gap to each critical EV, and will find there is at most only 1 critical EV that we can assign stations to. Therefore, in total we are able to assign the $N_w$ stations to at most $(N_w-1)+(N_w-2)+\dots+1=\frac {N_w(N_w-1)}{2}$ critical EVs, so as to satisfy \emph{\textbf{Lemma}~\ref{lem:critical}}.

This completes the proof.
\end{proof}

\end{document}